\documentclass{amsart}

\title{Classification of the limit shape for 1+1-dimensional FPP}
\author{Malte Ha{\ss}ler}
\address{Cornell University, Department of Mathematics, Malott Hall, 212 Garden Avenue, 14853 Ithaca, USA}
\email{mh2479@cornell.edu}

\usepackage{amsthm}
\usepackage{graphicx, xcolor}
\usepackage{amsmath, amsthm, amsfonts,amssymb, comment, bbm, enumitem, hyperref}
\newtheorem{Theorem}{}
\newenvironment{maintheorem}[1]{\renewcommand \theTheorem{#1}\Theorem}{\endTheorem}
\newtheorem{theorem}{Theorem}
\newtheorem{lemma}{Lemma}
\newtheorem{proposition}[lemma]{Proposition}
\newtheorem{corollary}[lemma]{Corollary}
\newtheorem*{remark}{Remark}

\newcounter{reminder}

\newcommand{\note}[1]{}
\newcommand{\hide}[1]{}

\theoremstyle{definition}
\newtheorem{definition}{Definition}

\newcommand{\eps}{\varepsilon}

\newcommand{\R}{\mathbb R}
\newcommand{\N}{\mathbb N}
\newcommand{\Z}{\mathbb Z}
\newcommand{\J}{J}

\newcommand{\vn}{{\lceil v n \rceil}}
\newcommand{\wn}{{\lceil w n \rceil}}
\newcommand{\vnx}{{\lceil (v+x) n \rceil}}

\newcommand{\deq}{\overset{d}{=}}

\newcommand{\Ent}{\operatorname{Ent}}
\newcommand{\B}{\mathcal B}
\newcommand{\mem}[1]{\text{\emph{#1}}}

\newcommand{\E}[2]{\mathbb{E}_{#1} \left[ #2 \right]}
\newcommand{\Prob}[1]{\mathchoice
{\mathbb P \left( \rule{0pt}{11pt} #1 \right)}%
{\mathbb P \left(#1\right)}%
{\mathbb P \left(#1\right)}%
{\mathbb P \left(#1\right)}%
}

\newcommand{\1}{{\mathbbm 1}}

\begin{document}

\maketitle

\begin{abstract}
We introduce a simplified model of planar first passage percolation where weights along vertical edges are deterministic. We show that the limit shape has a flat edge in the vertical direction if and only if the random distribution of the horizontal edges has an atom at the infimum of its support. Furthermore, we present bounds on the upper and lower derivative of the time constant. 
\end{abstract}

\section{Introduction}

Classical first passage percolation has a simple definition: Consider the $\Z^2$ lattice and put weights on each edge independently and identically according to some non-negative probability distribution. The graph distance now becomes a random metric space and one can study random metric balls at the origin. Cox and Durrett \cite{coxdurrett} have shown that under mild assumptions this random ball converges to a deterministic, convex limit shape, which depends on the choice of distribution. While it is conjectured that the limit shape is differentiable if the distribution is not deterministic and strictly convex if the distribution is continuous, the nature of the limit shape is still almost completely open. There are only a few special cases where one can say something about the limit shape, most notably:
\begin{itemize}
\item If the distribution is deterministic, then the limit shape is a dilation of the $\ell^1$-diamond.
\item If the distribution is supported on $[1,\infty)$ and has an atom at $1$ of mass at least $1/2$, then the limit shape has a flat edge near the main diagonal and is differentiable at the endpoints of that edge (see \cite{durrettliggett} and \cite{marchand}).
\item For each $k$, there exists $\eps>0$ such that a random variable supported on $[1,1+\eps]$ has a limit shape that is not a polygon with $k$ or less edges (see \cite[Theorem~1.4]{coalescence}).
\end{itemize}

Properties of the limit shape are of particular interest because they yield information about the geodesics (see \cite[Lemma~4.6]{coalescence}) and the existence of bi-infinite geodesics \cite{bigeodesics}, which has implications for the disordered Ising ferromagnet.  The interested reader can find more about results and conjectures in FPP in a survey by Auffinger, Damron and Hanson \cite{50years}. 

Given the difficulty of studying the limit shape, various simplified models have been introduced.  In 1998, Sepp\"{a}l\"{a}inen \cite{dir:seppalainen} studied a model where vertical edges have constant weights $1$ and the minimizing paths are restricted to be directed in the horizontal and vertical direction. This model is later known as the SJ-model due to major contributions by Johansson \cite{dir:joh}.  Further work has been done by O'Connell and Martin \cite{dir:oconnell}, \cite{dir:martin04}, \cite{dir:martin09}. They were able to exactly compute the limit shape for exponential, geometric and Bernoulli distributions as well as a combination of them (we like to point out that the first two distributions are the ones where last passage percolation is integrable). In our work however, we aim to show results in maximal generality instead of focusing on a few distributions. 

Let us formally introduce our model. It resembles the SJ-model, but one step closer to original FPP as paths are no longer required to be directed. To our knowledge, it has not been studied before. 

\subsection{The model and main result}

Consider the $\Z^2$ lattice and let $G$ be a probability measure on $[0,\infty)$. To each vertical edge $e$ of $\Z^2$ we assign a deterministic weight $\tau_e=1$. To each horizontal edge $e$ we assign a random weight $\tau_e$ independently sampled according to $G$. A \emph{path} $p$ is a finite or infinite sequence $p=(z_0,z_1,...)$ of points in $\Z^2$ such that $z_k$ and $z_{k+1}$ are connected by an edge for all $k\ge 0$. We say that an edge $e$ lies in $p$ if $e=(z_k, z_{k+1})$ for some $k$. 

For each path $p$ we assign a random \emph{passage time}
\[
T(p)=\sum_{e \in p} \tau_e.
\]

For $(x_1,y_1),(x_2,y_2) \in \Z^2$ define
\begin{equation}
\label{eq:inftime}
T((x_1,y_1),(x_2,y_2))=\inf_p T(p)
\end{equation}
where the infimum is taken over all paths starting at $(x_1,y_1)$ and ending at $(x_2,y_2)$.

For $\alpha \in \R$, let $\lceil \alpha \rceil$ denote the smallest integer greater or equal to $\alpha$. The first result is the existence of a time constant to which the random passage times for a particular direction converge to. 

\begin{theorem}
\label{Thm:time_const}
Let $G$ be any probability measure on $[0, \infty)$. Then there exists a convex, deterministic function $\Lambda: [0,\infty) \to [0, \infty)$ depending on $G$ such that 
\begin{equation}
\Lambda(\nu)=\lim_{n\to \infty} \frac 1 n T((0,0),(n,\vn)).
\end{equation}
almost surely and in $L^1$. 
\end{theorem}
The theorem is a standard result proved using the subadditive ergodic theorem. We will prove it in Section~\ref{Sec:appendix} together with the next lemma.

\begin{lemma}
\label{Lem:positive_time}
$\Lambda(v)=0$ if and only if $v=0$ and $G=\delta_0$.
\end{lemma}

\begin{definition}
\label{Def:limit_shape}
For $G\neq \delta_0$ and $0\le \theta \le \frac \pi 2$, we define the curve
\[
\B(\theta)=\frac 1 {\Lambda(\tan(\theta))\cos(\theta)}.
\]
The \emph{limit shape} $\B_G$ is the closed area embodied by the four symmetrical copies of the curve $\B(\theta)e^{i\theta}$, 
$\B(\theta)e^{-i\theta}, -\B(\theta)e^{i\theta}, -\B(\theta)e^{-i\theta}$. The union of the four curves then form the boundary of the limit shape $\partial \B_G$. 
\end{definition}

Moreover, we define the random balls:
\[
B(t)=\{x \in \Z^2: \, T((0,0),x)\le t \}.
\]

We can then state the analog of the Cox--Durrett limit shape theorem, which we will prove in Section~\ref{ss:limit_shape}. Note that unlike classical FPP, we no longer require a mean bound nor that $G(0)$ is larger than the critical probability for Bernoulli percolation. 

\begin{theorem}\label{Thm:limit_shape}
Let $G\neq \delta_0$. The limit shape $\B_G$ is convex, compact, nonempty and for each $\eps>0$
\[\Prob{ (1-\eps)\B_G \subset \frac{B(t)} t \subset (1+\eps)\B_G \, \text{ for all large } t} =1. \]
\end{theorem}

One can see that the points $(1/\Lambda(0), 0)$ and $(0,1)$ lie on the boundary of the limit shape, so by convexity, the curve $\B(\theta)e^{i\theta}$ cannot lie below the straight line between the points. For $x\in \partial\B_G$, we say that $x$ \emph{lies on a flat edge} if there exists $y \in \partial\B_G\setminus \{x\}$ such that the line segment from $x$ to $y$ is completely contained in $\partial\B_G$. 
We are now able to state the main theorem:

\begin{maintheorem}{Main Theorem}\label{Thm:classification}
Let $G\neq \delta_0$ and define $t_0:=\inf \{x: G([0,x])>0\}$. Then $(0,1)$ lies on a flat edge of the limit shape if and only if $G(\{t_0\})>0$. 
\end{maintheorem}

In particular, if $G$ has no atom at the infimum of its support, then the limit shape is not a polygon. Such a result is far from being known in the classical FPP model and the number of flat edges is useful for understanding the behavior of geodesics (see for example \cite{coalescence}). 

\subsection{Motivation and organization of the paper}

There has been recent progress in showing the differentiability of the limit shape by Yuri Bakhtin and Douglas Dow for discrete polymer models \cite{polymer} and Hamilton--Jacobi--Bellman equations \cite{hjb}, both of which can be interpreted as continuous versions of FPP. Their result is universal as it holds for a large class of distributions. 

Our model is initially motivated by the fact that it is a natural discretization of their polymer model in the zero temperature setting. While the discreteness prevents us from using the methods in \cite{polymer} to its full potential, we can still derive bounds on the upper and lower derivative of the time constant that link it to the local behavior of geodesics (Theorem~\ref{Thm:diffbounds}). The main idea here is to replace continuous shear maps, which are not applicable anymore in our discrete model, by a random analog that is supported over the integers. 

We will present and derive these inequalities in the following section after introducing basic properties of geodesics. The inequalities are not essential for us to prove the main theorem, and depending on the reader's interest one can skip forward to Section~3. However, the concepts presented appear to be rather new in the field and provide some intuition why the main theorem holds.

\section{Connecting the derivative of the time constant with the local geometry of geodesics}

\subsection{Geodesics}

A path $p$ from $(x_1,y_1)$ to $(x_2,y_2)$ is a  \emph{geodesic} if it assumes the infimum in \eqref{eq:inftime}.

We call a path $p$ \emph{(right) semi-directed} if
\[
z_{k+1}-z_k \in \{(1,0), (0,1), (0,-1) \} \quad \forall \, k \ge 0. 
\]
In other words, the path $p$ is allowed to move up, down and right, but not left. 

\begin{proposition}
Suppose $\gamma$ is a geodesic from $(x_1,y_1)$ to $(x_2,y_2)$ with $x_2 \ge x_1$. Then there exists a semi-directed, non-intersecting geodesic $\gamma'$ from $(x_1,y_1)$ to $(x_2,y_2)$.
\end{proposition}
\begin{proof} 
Erasing loops cannot increase the passage time, so we can clearly obtain a non-intersecting geodesic. Suppose $\gamma$ was not semi-directed already. Then there exists $x_3, y_3 \in \Z$ and $k \in \N$ such that $z_k=(x_3,y_3)$ and $z_{k+1}=(x_3-1, y_3)$. The first time $\gamma$ hits the vertical line $\{x_2\}\times \Z$, we can take a straight line to $(x_2,y_2)$ and the resulting path is a geodesic because the edge weights are non-negative. Thus, we may assume that $x_3<x_2$. Then there exists $k'>k+1$ such that $z_{k'}=(x_3,y_4)$ lies on the vertical line $\{x_3\}\times\Z$. Let $\gamma'$ be the path that goes from $z_k$ straight up or down to $z_{k'}$. Since $T(z_k,z_{k'})=|y_4-y_3|$, the path $\gamma'$ is a geodesic. By repeating this step every time a left-turn occurs, we obtain a semi-directed geodesic. 
\end{proof}

From the proof it is also clear that if $G(0)=0$, then almost surely every geodesic is semi-directed and non-intersecting. 

\begin{proposition}
\label{Prop:geodesics}
For all $m\in \Z$ and $n \in \Z_{\ge 0}$, there always exists a geodesic connecting the origin and $(n,m)$ and the number of such geodesics is finite. 
\end{proposition}
\begin{proof}
Define the cylinder sets
\[
C_n(h):=\{ (k,j) \in \Z^2: k \in \{0,...,n\}, \, |j| \le  h\}.
\]
Pick a semi-directed path $p$ going from the origin to $(n,m)$ and choose $h \in \N$ large enough such that $p \in C_n(h)$ and $T(p)\le h$. Let $p'$ be any semi-directed path going from the origin to $(n,m)$ that is not completely contained in $C_n(h)$. Then $T(p')>h \ge T(p)$, so $p'$ is not a geodesic. Since the number of semi-directed, non-intersecting paths from the origin to $(n,m)$ which are completely contained in $C_n(h)$ is finite, the minimum of their passage times is assumed. 
\end{proof}

\begin{proposition}
If the distribution $G$ is continuous, geodesics are almost surely unique.
\end{proposition}
\begin{proof}
For $s,t \ge 1$, let $G(s,t)$ be the distribution of $X_1+...+X_s-Y_1-...-Y_t$, where $X_1,...,X_s, Y_1,...,Y_t$ are i.i.d.\ random variables with distribution $G$. If $G$ is continuous, so is $G(s,t)$. Now let $Z_1,Z_2,..$ be an enumeration of the edge weights. Suppose there exists two geodesics, then there exists a finite subset $\{i_1,...,i_s,j_1,...,j_t\}\subset \N$ such that $Z_{i_1}+...+Z_{i_s}-Z_{j_1}-...-Z_{j_t}=0$. Since $G(s,t)$ is continuous, for each choice of indices $\{i_1,...,i_s,j_1,...,j_t\}$, this event has zero probability. Since we sum over a countable set of indices, geodesics are almost surely unique.
\end{proof}

For the rest of the paper, we will speak of \emph{the} geodesic even if $G$ is not continuous. We mean with this that a geodesic has been chosen in some deterministic way. 

\subsection{The derivative of the time constant}

To state and prove the theorem on derivative inequalities, we need to introduce the concept of pioneer points:

\begin{definition}
Let $p$ be a semi-directed path from the origin to a point $(n,m)$. For $k=0,...,n$ let $f_p(k)\in \Z$ be such that $(k, f_p(k)) \in \Z^2$ is the first intersection of $p$ with the vertical line $\{k\}\times \Z$. In particular, $f_p(0)=0$. The points $(k, f_p(k)) \in \Z^2$ are referred to as \emph{pioneer points}. 
We also introduce the convention $f_p(n+1)=m$ (even though $(n+1,m) \not \in p$). 
\end{definition}

\begin{definition}\label{Def:urd}
Let  $p$ be a semi-directed path from the origin to a point $(n,m)$. The number of \emph{right-turns} of $p$ is defined as 
\[
R(p):=\left| k \in \{0,...,n\} : \, f_p(k)=f_p(k+1)\right|.
\] 
similarily, we define the number of \emph{up-turns} by
\[
U(p):=\left| k \in \{0,...,n\} : \, f_p(k)<f_p(k+1)\right|,
\] 
and the number of \emph{down-turns} by
\[
D(p):=\left| k \in \{0,...,n\} : \, f_p(k)>f_p(k+1)\right|.
\] 
\end{definition}

By symmetry of the model, we can meaningfully extend the time constant to $\R$ and $\Lambda(-v)=\Lambda(v)$. Since the function $\Lambda: \R \mapsto [0,\infty)$ is convex, in each direction $v \in \R$ we have upper and lower derivatives

\[
\partial^+ \Lambda(v)=\lim_{w\searrow v} \frac{\Lambda(w)-\Lambda(v)}{w-v} \quad\text{ and }\quad \partial^- \Lambda(v)=\lim_{w\nearrow v} \frac{\Lambda(v)-\Lambda(w)}{v-w}.
\]
The time constant is differentiable at $v$ if and only if $\partial^+ \Lambda(v)=\partial^- \Lambda(v)$. For $v=0$, due to symmetry, this is equivalent to $\partial^+ \Lambda(0)=0$. 

\begin{proposition}
\label{prop:diffbounds}
We have $|\partial^{+/-} \Lambda(v)| \le 1$ and for $v>0$ also $\partial^{+/-}\Lambda(v)\ge 0$. 
\end{proposition}
\begin{proof}
Since we can always connect a geodesic going from the origin to $(n, \vn)$ by a straight vertical line to $(n, \wn)$, we have \[|T((0,0),(n,\vn))-T((0,0),(n,\wn))| \le |\wn- \vn|\] and thus $|\Lambda(v)-\Lambda(w)|\le |v-w|$. This shows $|\partial^{+/-} \Lambda(v)| \le 1$. The other bound follows from convexity and symmetry.
\end{proof}

Our second main result is an improvement on the bounds in the previous proposition which links the derivative to the local geometry of the geodesics:

\begin{theorem}
\label{Thm:diffbounds}
Let $G$ be non-deterministic and $v\ge 0$. For each $n$, let $\gamma^n(v)$ be the geodesic from the origin to $(n, \vn)$. The following bounds hold almost surely:
\begin{equation}
\label{ineq:upperdiff}
\partial^+ \Lambda(v) \le \liminf_{n \to \infty} \frac 1 n \left( U(\gamma^n(v))+R(\gamma^n(v))-D(\gamma^n(v))\right)
\end{equation}
and
\begin{equation}
\label{ineq:lowerdiff}
\partial^- \Lambda(v) \ge \limsup_{n \to \infty} \frac 1 n \left( U(\gamma^n(v))-R(\gamma^n(v))-D(\gamma^n(v))\right).
\end{equation}
\end{theorem}

Since $U(\gamma^n(v))+R(\gamma^n(v))+D(\gamma^n(v))=n+1$, this again yields the bound from the Propositon~\ref{prop:diffbounds}, although it might not always be the case that the expression in \eqref{ineq:lowerdiff} is positive. If we can show that down-turns in geodesics occur with a linear ratio, then \eqref{ineq:upperdiff} tells us that $\partial^+\Lambda(v)<1$ which implies that the limit shape has no flat edge at $(0,1)$. 

In classical FPP the study of the local geometry of geodesics has recently caught attention. We refer here to an article by Jacquet \cite{pattern} who shows that 
any positive probability event only depending on the weights in a finite box will be observed along a geodesic a linear ratio of times with respect to the distance between the endpoints of the geodesic. In particular, this applies to down-turns. The main idea is to increase edge weights inside the box to force the geodesic into a specific shape. Unfortunately, this technique does not apply to our model because the vertical edge weights are static. So while we employ a different strategy to proof the main theorem, Theorem~\ref{Thm:diffbounds}  gives some intuition why it should hold.

However, a simple application of this result is that $\partial^+ \Lambda(0)<1$ if $G$ is not deterministic and hence the limit shape is not an $\ell^1$-diamond. This is because it is very unlikely for a geodesic to have a lot of consecutive right-turns, hence $R(\gamma)/n < 1$ and symmetry at $v=0$ yields $\E{}{D(\gamma)}=\E{}{U(\gamma)}$. 

A natural guess for an expression for the derivative based on Theorem~\ref{Thm:diffbounds} is $\Lambda'(v)= \lim_{n\to \infty} \E{}{U(\gamma^n(v))-D(\gamma^n(v))}/n$.
However, this equality is not true in general: If $G(0)>1/2$, we can move horizontally along zero-weight edges and move upwards if we encounter a larger weight. If $v$ is large enough, we thus find a path with minimal passage time $vn$ with large probability. So $\Lambda'(v)=1$, but $\E{}{R(\gamma^n)}\ge n/2$.

\begin{figure}[h]
\begin{center}
\begin{minipage}[c]{\textwidth}
\includegraphics[width=\textwidth]{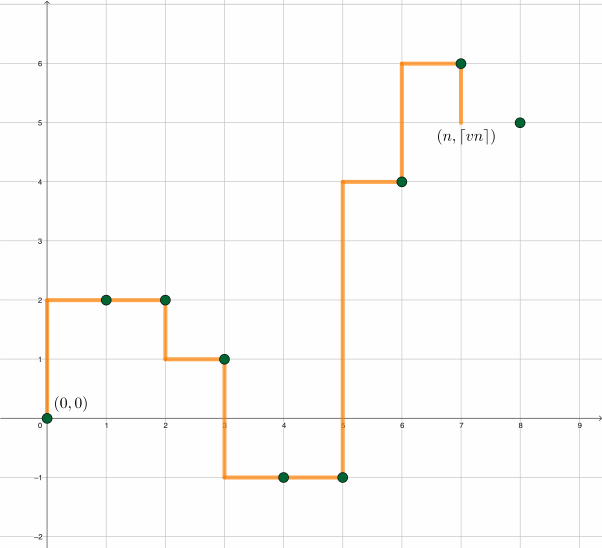}
\end{minipage}
\end{center}
\caption{A semi-directed path $p$ from the origin to $(n, \vn)$ with $n=7$ and $v=0.7$. Pioneer points are shown in green and $f_p=(0,2,2,1,-1,-1,4,6,5)$. }\label{Fig:pioneer} 
\end{figure}

\subsection{An equivalent model}\label{Sec:equiv_model}

We now rewrite our model in a way that resembles the polymers studied in \cite{polymer}.

For each semi-directed, non-intersecting path $p$ from the origin to $(n,\vn)$, define
\[
\gamma=\gamma(p)=(f_p(0),...,f_p(n+1)).
\]
See Figure~\ref{Fig:pioneer} for an example.  Moreover let $V(x)=|x|$ denote the absolute value, let $F_k(x)$ be the weight of the edge $((k-1),x),(k,x))$ and write $\Delta_k \gamma=\gamma_{k+1}-\gamma_k$. We rewrite the passage time as a sum of vertical edges and a sum of the horizontal edge weights:

\begin{equation}
\label{alt:time}
T(p)=A^n(\gamma):= \sum_{k=0}^n V(\Delta_k \gamma) + \sum_{k=1}^n F_k(\gamma_k).
\end{equation}
Then $\gamma(p)\in \Gamma^n(v)$ where

\[
\Gamma^n(v)=\{\gamma \in \Z^{n+2}: \, \gamma_0=0, \gamma_{n+1}=\lceil vn \rceil \}.
\]

Conversely, one can check that given $\gamma  \in \Gamma^{n}(v)$, there exists a unique semi-directed, non-intersectng path $p$ from the origin to $(n,\vn)$ such that $\gamma=\gamma(p)$. Thus, we also have
\begin{equation}
\label{altinf}
T((0,0),(n,\vn))=A^n_*(v):=\inf_{\gamma \in \Gamma^n(v)} A^n(\gamma)
\end{equation}
and almost surely and in expectation
\[
\Lambda(v)=\lim_{n \to \infty}  \frac 1 n A^n_*(v). 
\]

We say $\gamma \in \Gamma^n(v)$ is a geodesic if the infimum in \eqref{altinf} is attained. Clearly, a non-intersecting, semi-directed path $p$ is a geodesic in the original model if and only if $\gamma(p)$ is a geodesic. So existence and finiteness of geodesics (Proposition~\ref{Prop:geodesics}) also holds here. 

\subsection{Random shearing}

The main idea in \cite{polymer} is to shift the environment by so-called shear maps of the form $(\Xi_v \gamma)_k= \gamma_k + vk $. The obstacle is that our paths $\gamma$ live in $\Z$, so this map is not well-defined for all $v$ and $k$. We overcome this issue by introducing additional randomness.

For $x \in [0,1]$, let $\omega(x) \in \{0,1\}^\infty$ be a sequence of independent $Ber(x)$-distributed random variables, i.e.\  $\Prob{\omega_i=1}=x$ for $i=1,2,\dots$. Let $\Omega(x)$ denote the probability space we sample $\omega(x)$ from. 

For each $\omega \in \Omega(x)$ and path $\gamma \in \Z^{n+2}$, define the shear maps coordinate-wise for $k=0,1,\dots$
\[
(\Xi_{\pm \omega} \gamma)_k := \gamma_k \pm \sum_{i=1}^{k} \omega_i.
\]
The  inverse operator is given by $\Xi_\omega^{-1} = \Xi_{-\omega}$. Note that in expectation they behave analogously to the shear maps in \cite{polymer}:
\[
\E{}{(\Xi_\omega \gamma)_k}=\gamma_k + \sum_{i=1}^{k} \E{}{\omega_i}=\gamma_k + kx .
\]
Also define the shears on the random environment
\[
(\Xi^{*}_{\pm \omega} F)_k (x) = F_k( x \pm \sum_{i=1}^{k} \omega_i).
\]
By independence of the environment $F=\{F_k(x)\}$, we get that the shear maps are measure preserving for every $\omega$, in particular we have for any $\omega$ and jointly for all paths $\gamma$ and all $n$ 
\begin{equation}
\label{eq:measure_pres}
A^n(\gamma)(F)\deq A^n(\gamma)(\Xi^*_{\pm \omega}F). 
\end{equation}

We define the sheared passage time:
\[
B^n(\pm \omega)(\gamma)=\sum_{k=0}^n V(\Delta_k \gamma \pm \omega_{k+1}) + \sum_{k=1}^n F_k(\gamma_k),
\]
and consider the minimization problem:
\begin{equation}
\label{eq:bmin}
B^n_*(v,\pm \omega)=\inf \{ B^n(\pm \omega)( \gamma): \, \gamma\in \Gamma^n(v) \}. 
\end{equation}
In particular, if $x=0$, then $\omega\equiv 0$ and $B^n(\omega)=A^n$. Note that
\begin{equation}
\label{eq:ab}
B^n(\gamma)(F)=A^n(\Xi_{\omega}\gamma)(\Xi^*_{-\omega}F). 
\end{equation}

We conclude that we can obtain the time constant at different directions from the sheared passage times.

\begin{lemma}[Almost sure convergence of sheared passage times]\label{Lem:asB}
For any $x\in [0,1]$, let $\E{\Omega}{\,\cdot\,}$ denote the expectation with respect to $\Omega(x)$. Then almost surely
\[
\lim_{n \to \infty} \frac 1 n \E{\Omega}{B^n_*(v,\pm \omega)}=\Lambda(v\pm x).
\]
\end{lemma}
\begin{proof}
By \eqref{eq:measure_pres} and \eqref{eq:ab}, the passage times $B^n(\omega)(\gamma)$ and $A^n(\Xi_\omega \gamma)$ have the same distribution jointly in $n$ and for any paths $\gamma$. The map $\Xi_\omega$ is a bijection from $\Gamma^n(v)$ to the set of semi-directed paths $\gamma$ with $\gamma_0=0$ and $\gamma_{n+1}=\vn+\sum_{i=1}^{n+1} \omega_i$. Hence, we can take the minimum over all paths and obtain that the following sequences of random variables have the same joint distribution
\begin{equation}
\label{distreq}
\left( B^n_*(v,\omega)\right)_{n\ge 1} \deq \left( T((0,0),(n,\vn+\sum_{i=1}^{n+1} \omega_i)) \right)_{n\ge 1}.
\end{equation}
Consider the event that the sum of the $\omega_i$ is close to its mean:
\[
\mathcal S_n := \{ \omega \in \Omega(x):\, (n+1)x-(n+1)^{2/3}\le \sum_{i=1}^{n+1} \omega_i \le (n+1)x+(n+1)^{2/3} \}. 
\]
By Hoeffding's lemma, $\Prob{\mathcal S_n^c}\le 2\exp(-2(n+1)^{1/3})$. Thus, almost surely the event $S_n$ occurs for sufficiently large $n$ by Borel-Cantelli. Under $\mathcal S_n$, the joint distribution on the right of \eqref{distreq} is stochastically dominated by $(A^n_*(v+x)+(n+1)^{2/3})_{n}$ by subadditivity. Since $(A^n_*(v+x)+ n^{2/3})/n$ converges to $\Lambda(x+v)$ almost surely, it is also an almost sure upper bound for the the limit superior of $B^n_*(v,\omega)/n$. The argument for the lower bound and the negative case is similar. 
\end{proof}

\subsection{The upper derivative}

The advantage of the sheared passage times is that the set of admissible paths $\Gamma^n(v)$ no longer depends on the direction. We will now use this fact to show \eqref{ineq:upperdiff}. For $x\in [0,1]$, sample $\omega$ from $\Omega(x)$.  For each $z \in \Z$ and $i \in \N$, we have the identity
\[
V(z+\omega_i)-V(z)=\1(\omega_i=1) \Delta V(z)
\]
where 
\begin{equation}
\label{eq:DeltaV}
\Delta V(z)=V(z+1)-V(z)=|z+1|-|z|=\begin{cases} +1& , \, z\ge 0 \\ -1& , \, z<0 \end{cases}
\end{equation}

is the discrete (right) derivative of the absolute value. It follows that for any path $\gamma \in \Gamma^n(v)$:

\begin{equation}
\label{eq:timediff}
\frac 1 n B^n(\omega)(\gamma) - \frac 1 n A^n(\gamma)=\frac 1 n \sum_{k=0}^n \1(\omega_{k+1}=1)\Delta V(\Delta_k \gamma).
\end{equation}

Now let $\gamma$ be the geodesic for the unsheared problem, we then have $A^n(\gamma)=A_*^n(v)$ and $B^n(\omega)(\gamma)\ge B_*^n(v,\omega)$. Hence, \eqref{eq:timediff} implies

\begin{equation*}
\frac 1 n B^n_*(v,\omega)- \frac 1 n A^n_*(v) \le \frac 1 n \sum_{k=0}^n \1(\omega_{k+1}=1)\Delta V(\Delta_k \gamma).
 \end{equation*}
 
 Since the shear maps are independent of the edge weight environment, taking the expectation with respect to $\Omega(x)$ yields
 \[
 \frac 1 n \E{\Omega}{B^n_*(v,\omega)}- \frac 1 n A^n_*(v) \le \frac x n \sum_{k=0}^n \Delta V(\Delta_k \gamma).
 \]
 We may rewrite the right-hand side using Definition~\ref{Def:urd} and \eqref{eq:DeltaV}. Taking the limit $n\to \infty$,  by Lemma~\ref{Lem:asB} we conclude that almost surely
 \[
 \Lambda(v+x)-\Lambda(v)\le x \liminf_{n\to \infty} \frac 1 {n} {U(\gamma^n(v))+R(\gamma^n(v))-D(\gamma^n(v))}.
 \]

\subsection{The lower derivative}

The lower bound calculation is very similar. We now have
\[
V(z)-V(z-\omega_i)=\1(\omega_i=1)\Delta V(z-1)=\begin{cases} +\1(\omega_i=1)& , \, z> 0 \\ -\1(\omega_i=1)& , \, z\le0. \end{cases}
\]

For any path $\gamma \in \Gamma^n(v)$ we have the identity
\[
\frac 1 n A^n(\gamma)- \frac 1 n B^n(-\omega)(\gamma)=\frac 1 n \sum_{k=0}^n \1(\omega_{k+1}=1) \Delta V(\Delta_k \gamma -1).
\]
Setting $\gamma$ to be the geodesic yields
\[
\frac 1 n A^n_*(v)-\frac 1 nB^n_*(v,-\omega)\ge \frac 1 n \sum_{k=0}^n \1(\omega_{k+1}=1)\Delta V(\Delta_k \gamma -1).
\]
Taking the expectation with respect to $\Omega(x)$ and the limit then yields that almost surely
\[
\Lambda(v)-\Lambda(v-x)\ge x \limsup_{n\to \infty} \frac 1 {n} {U(\gamma^n(v))-R(\gamma^n(v))-D(\gamma^n(v))}.
\]

\begin{remark}
There is great freedom in defining the shear map, although it is unclear if it will give improved results. If we want to ensure that for $\gamma \in \Gamma^n(v)$, the path $\Xi_\omega \gamma$ is always in $\Gamma^n(v+x)$, one can define weakly-dependent shear maps $\omega^n(x)=(\omega_1^n(x),\dots,\omega_n^n(x))$ as follows: Uniformly sample a permutation $\sigma_{n+1}$ of the integers $\{1,\dots,n+1\}$ and define
\begin{equation*}
\omega^{n,v}(x)_i:=\begin{cases} 1 &\text{ if } \sigma_{n+1}^{-1}(i)\le \vnx - \vn \\ 0 &\text{ if }  \sigma_{n+1}^{-1}(i)>\vnx - \vn  \end{cases}, \quad i\in\{1,...n+1\}.
\end{equation*}
Then $\sum_{i=1}^{n+1} \omega^{n,v}(x)_i=\vnx - \vn$ and $\E{}{\omega^{n,v}(x)_i}\sim x$ for large $n$. 
\end{remark}

\section{Classification of the limit shape}

This section deals with the proof of the \ref{Thm:classification}. We will work with different first passage models, so let us first introduce some notation. For a probability measure $G$, let $T^G(\cdot,\cdot)$ and $\Lambda^G(v)$ denote the passage time and the time constant in our undirected model with edges distributed according to $G$. 

For the directed SJ-model, we write $T^G_{dir}(\cdot,\cdot)$ and $\Lambda^G_{dir}(v)$. If the underlying distribution is clear, we omit the superscript. 

When $G$ is a $\{0,1\}$-Bernoulli distribution with $G(0)=p$, let $T^p_{FPP}(\cdot,\cdot)$ and $\Lambda^p_{FPP}(v)$ be the passage time and time constant with these Bernoulli weights in the original, undirected FPP model where both horizontal and vertical edges are randomized. 

Finally, we consider a model where the \emph{sites} of the $\Z^2$ lattice have i.i.d.\ Bernoulli weights and paths are required to be semi-directed. We refer to this model as the \emph{site percolation model} and the passage times are denoted by $T^p_{site}(\cdot,\cdot)$.

\subsection{Regularity of the passage time}

The following lemma explains why unlike the standard FPP models, the existence of the limit shape and many other properties hold in our model without moment assumptions. The main reason is that we can move around edges with high weights at a controlled cost.

\begin{definition}Let $G$ be non-deterministic and $B>t_{0}(G)$. For any edge $e$ of the lattice, we define the \emph{detour height} as the distance to an edge with small passage time:
\[
k_B(e):=\min\left (k\ge 1 | \,  \min( \tau(e+(0,k)),\tau(e-(0,k)))< B\right).
\]
The \emph{detour} of $e$ is the path that starts at one edge of $e$, goes $k_B(e)$ edges upwards or downwards, passes the horizontal edge with weight less than $B$ and goes $k_B(e)$ edges down back to the other endpoint of $e$. The \emph{detour time} $dt_B(e)$ is the passage time of the detour.
\end{definition}

\begin{lemma}
\label{Lem:expomoment}
Set $e_0:=((0,0),(1,0))$ and $X=dt_B(e)$ for some $B>t_0(G)$. Then there exists $\theta >0$ such that $\E{}{e^{\theta X}}<\infty$. \end{lemma}
\begin{proof}
We have the bound $X\le 2k_B(e_0)+ B$. By independence of edge weights, $k_B(e)$ has a geometric distribution. So $X$ also has exponential tails and thus finite exponential moments.
\end{proof} 

As a consequence, passage times always have finite moments

\begin{corollary}
\label{cor:passagemoment}
For any $G$, there exists $\theta>0$ such that given $x_1,x_2,y_1,y_2 \in \Z$ and $T=T((x_1,y_1),(x_2,y_2))$ we have $\E{}{e^{\theta T}}<\infty$.
\end{corollary}
\begin{proof}
For $i\in \{x_1,...,x_2-1\}$, set $e_i=((i, y_1),(i+1,y_1))$. Then $T\le dt_B(e_{x_1})+\dots + dt_B(e_{x_2-1})+|y_2-y_1|$. All these summands are independent. 
\end{proof}

We continue with a statement going back to Hammersley and Welsh \cite{hamwelsh}:

\begin{lemma}[Stochastic dominance]
\label{Lem:partial_order}
Let $F_1$ and $F_2$ be two distribution functions satisfying $F_1(x)\le F_2(x)$ for all $x\ge 0$. Then the associated time constant functions satisfy $\Lambda_2(v)\le \Lambda_1(v)$ for all $v$. 
\end{lemma}
\begin{proof}
For every horizontal edge, we can create a coupling such that the edge weight $\tau_1$ has $F_1$, the edge weight $\tau_2$ has distribution $F_2$ and $\tau_2\le \tau_1$. Hence, the passage times are ordered as well.
\end{proof}

The following lemmas deal with the continuity of the time constant under weak convergence but we only need it for some special cases. The results are known in classical FPP. 

\begin{lemma}[Continuity for Bernoulli distributions]
\label{Lem:cnt_ber}
Let $\Lambda^p$ be the time constant function with $\{0,1\}$-Bernoulli distribution (i.e.\ $G(0)=p$). Then for every $v$ pointwise $\lim_{p \to 0} \Lambda^p(v)=\Lambda^0(v)=v+1$. The same holds for $\Lambda_{dir}$.
\end{lemma}
\begin{proof}
Note that we have
\[
\Lambda_{FPP}^p(v)\le \Lambda^p(v) \le v+1.
\]
By Theorem~1.14 in \cite{cox}, we know that $\Lambda_{FPP}^p(v)$ converges to $\Lambda_{FPP}^0(v)=v+1$, so the same must hold for our model and the SJ-model.\end{proof}

Showing this for arbitrary distributions is a bit harder, we will move its prove to Section~\ref{ss:cont}.

\begin{lemma}[Continuity for truncations]
\label{Lem:continuity}
For a distribution function $G(x)$ and $B>0$, let $G^B$ be the distribution obtained by truncating above at $B$, i.e.
\begin{align*}
G^B(x)=\begin{cases} G(x) \, \text{ if } x<B,\\ 1 \, \text{ if } x \ge B \end{cases}.
\end{align*}
And let $\Lambda$ and  $\Lambda^B$ be the associated time constant functions. Then $\Lambda^B$ converges to $\Lambda$ pointwise as $B \to \infty$. 
\end{lemma}

We will also need some knowledge of large deviations of the passage time. Again, similar results are known for FPP and the proofs are found in Sections~\ref{ss:right_dev} to \ref{ss:left_dev}. Contrary to the original model however, no moment assumptions are required, which in parts is explained in Lemma~\ref{Lem:expomoment}.

\begin{theorem}[Large deviations, right tail]
\label{Thm:upper_deviation}
For each $\eps>0$ and $v_0\ge 0$, there exist constants $A,B$  such that for all $n\ge 1$ and $v\in [0,v_0]$:
\[
\Prob{T((0,0),(n,\vn))\ge n (\Lambda(v)+\eps)}\le A \exp(-Bn). 
\]
\end{theorem}

This estimate also gives us a tail bound on the number of edges in a geodesic:

\begin{corollary}[Length of geodesics]
\label{Cor:length}
For each $\eps>0$, there exists constants $A,B$ depending on $\eps$ such that for all $n\ge 1$ and $v$:
\[
\Prob{|\gamma^n(v)|\ge n (\Lambda(v)+1+\eps)}\le A \exp(-Bn). 
\]
\end{corollary}
\begin{proof}
Simply note that $|\gamma^n(v)|-n$ equals the number of vertical edges of $\gamma^n(v)$, which have weight $1$. Hence, the event $|\gamma^n(v)|\ge n (\Lambda(v)+1+\eps)$ implies the event of Theorem~\ref{Thm:upper_deviation}.
\end{proof}

\begin{theorem}[Large deviations, left tail]
\label{Thm:lower_deviation_dir}
 For each $\eps>0$ and $v_0\ge 0$, there exist constants $A,B$ such that for all $n\ge 1$ and $v\in [0,v_0]$:
\[
\Prob{T((0,0),(n,\vn))\le n (\Lambda(v)-\eps)}\le A \exp(-Bn). 
\]
The same holds for the SJ model when $G$ corresponds to a Bernoulli distribution.
\end{theorem}

For the site percolation model, we will need a combination of  continuity and large deviation results, to be proven in Section~\ref{ss:spm}. 

\begin{lemma}[Site percolation]
\label{Lem:spm}
For every $\eps>0$ and $v\in \R$, there exists constants $A,B>0$ and $p_\eps>0$ such that for all $p<p_\eps$ we have
\[
\Prob{T_{site}^p((0,0),(n,\vn))\le n (v+1-\eps)}\le A \exp(-Bn).
\]
\end{lemma}

\subsection{ Bounds on the limit shape} 

Sepp\"{a}l\"{a}inen's main result can be reformulated as the following:

\begin{theorem}
\label{Thm:exact}
Let $p\in [0,1]$ and $\kappa > \lambda \ge 0$. Let each edge $e$ be independently distributed according to
\[ 
\Prob{ e = \lambda}=p, \, \Prob{e=\kappa}=q:=1-p.
\]
Then
\begin{equation*}
\Lambda_{dir}(v)=\begin{cases} \lambda + v \text{ if } v>q/p \\ \lambda+v+(\kappa-\lambda)(\sqrt{q}-\sqrt{pv})^2 \text{ if } v \le q/p  \end{cases}
\end{equation*}
\end{theorem}

We can derive from this the limiting behavior of the time constant in our model.

\begin{proposition}
\label{Prop:asymp}
The line $v+t_0$ is an asymptote of the time constant function, i.e.\ $\lim_{v \to \infty} \Lambda(v)-v-t_0=0$. If $G$ have an atom at $t_0(G)$, then we even have $\Lambda(v)=v+t_0$ for all $v\ge  (1-G(t_0))/G(t_0)$. 
\end{proposition}
\begin{proof}
Let $B>t_0+1$. We first prove the proposition for the truncated passage time $\Lambda^B(v)$. For $\eps\in(0,1)$, set $p=G(t_0+\eps)$. Let $G_p$ be the Bernoulli distribution
\[
\Prob{ e=t_0+\eps}=p, \, \Prob{e=B}=1-p
\]
Then for $v>(1-p)/p$ we have by Theorem~\ref{Thm:exact},
\[
\Lambda^B(v)\le \Lambda^{G_p}(v) \le \Lambda^{G_p}_{dir}(v)\le t_0+\eps+v. 
\]
Since we also have the trivial lower bound $\Lambda(v)\ge t_0+v$, we conclude that  $\Lambda^B(v)-v-t_0$ converges to zero uniformly in $B$ as $v\to \infty$. By Lemma~\ref{Lem:continuity} and the Moore--Osgood theorem, the same holds for $\Lambda(v)$. Also if $G(t_0)>0$, we may choose $\eps=0$ in the estimate above and must reach equality at finite $v$. 
\end{proof}
Note that the same clearly holds for $\Lambda_{dir}(v)$ if the distribution is bounded. The major work left is to derive a lower bound on $\Lambda(v)$ to show that the function stays strictly above the asymptote if $G$ has no atoms at the infimum of its support. We will prove it in Section \ref{subsec:above_asymp}.

\begin{proposition}
\label{Prop:above_min}
Let $G$ have no atom at $t_0$. For all $v\ge 0$, we have $\Lambda(v)>t_0+v$.
\end{proposition}

\textit{Proof of \ref{Thm:classification} using Proposition~\ref{Prop:above_min}:} We need to show the equivalence of the following
\begin{enumerate}
\item[(1)] The limit shape has a flat edge at $(0,1)$. 
\item[(2)] There exists $v_0$ such that $\Lambda(v)=v+t_0$ for all $v\ge v_0$. 
\end{enumerate}
Set $x=\cos(\theta)\B(\theta), y=\sin(\theta)\B(\theta)$, then we rewrite $\B(\theta)e^{i\theta}=(x,y)$. Assumption (1) means that there exists $(x_0,y_0)\in \partial \B$ such that the straight line from $(0,1)$ to $(x_0,y_0)$ forms the boundary of the limit shape. Thus for all $x\in (0,x_0)$ we have
\[
 y(x)=1-\frac{1-y_0}{x_0}x
\]
Using the identities $v=y/x$ and $\Lambda(v)=1/x$ coming from Definition~\ref{Def:limit_shape}, this becomes
\[
\Lambda(v)=v+\frac{1-y_0}{x_0} \quad\, \forall v\ge v_0:=y_0/x_0.
\]
Finally by Proposition~\ref{Prop:asymp}, we must have $t_0=(1-y_0)/x_0$. Conversely, (2) translates into $y(x)=1-t_0x$ for all $x<x_0:=1/\Lambda(v_0)$.

\subsection{The local behavior of semi-directed paths}

There are two different ways we can ensure that the time constant is strictly above $v+t_0$. Firstly, if the geodesic to $(n,\vn)$ moves downwards along $\eps n$ vertical edges, then the passage is at least $(v+\eps+t_0)n$. Secondly, if the geodesic is directed in a large box, we can compare its passage time to the directed model where we know by Theorem~\ref{Thm:exact} that the time constant is above the theoretical minimum if the slope in the box is not too large. Thus, before proving Proposition~\ref{Prop:above_min}, we will derive some deterministic results that a semi-directed path has to move downwards a significant amount of times or has some local regions where the slope of the path is bounded. 

\begin{lemma}
\label{Lem:ftc}
Let $f(x)$ be a Riemann-integrable function on $[0,1]$ and set $L=\int_0^1 f(x) \,dx$. For $\eps>0$, define the sets
\[
D=\{x\in [0,1]:\, f(x)<0 \} , \quad L_\eps=\{ x\in [0,1]: f(x)\in [0,L+\eps] \}. 
\]
Then one of the following holds:
\begin{equation}
\label{ineq:Dint}
\left| \int_D  f(x) \, dx \right| \ge \frac \eps 2.
\end{equation}
or
\begin{equation}
\label{ineq:slope}
\mu(L_\eps \cup D) \ge \frac \eps {2(L+\eps)}. 
\end{equation}
where $\mu(\cdot)$ denotes the Lebesgue measure. 
\end{lemma}
\begin{proof}
If \eqref{ineq:Dint} does not hold, then we have
\[
L=\int_D f(x)\, dx + \int_{L_\eps} f(x)\,dx + \int_{(L_\eps \cup D)^c} f(x)\, dx \ge - \frac \eps 2 + \mu((L_\eps \cup D)^c) (L+\eps)
\]
and hence
\[
\mu(L_\eps \cup D)\ge 1-\frac{L+\eps/2}{L+\eps}=\frac \eps {2(L+\eps)}.
\]
\end{proof}

\begin{corollary}
\label{Cor:ftc}
Let $p$ be a semi-directed path from the origin to $(n,\vn)$ and let $\eps>0$. For each proper divisor $K$ of $n$, divide the cylinder $[0,n]\times \Z$ into $n/K$ strips and count the ratio $\eta$ of strips in which the slope of $p$ is bounded from above:
\[
\eta:=\frac K n \left|\left\{ i\in \{0,...,n/K-1\}: \, f_p^*(K(i+1))-f_p^*(K i)\le K(v_n+\eps) \right\}\right|
\]
where $v_n=\vn/n$ and $f_p^*(x)=f_p(x)$ as defined in Section~\ref{Sec:equiv_model} with the exception of $f_p^*(n)=(n,\vn)$. Then $p$ moves downwards along at least $\eps n/(2K)$ vertical edges or $\eta > \eps/ (2(v_n+\eps))$. 
\end{corollary}
\begin{proof}
Let $F(x)$ be the piecewise linear function connecting the points $f_p^*(Ki)$ for $i=0,...,n/K$ in that order and let $f(x)$ be the derivative of $F(x)$ wherever it is differentiable. Now apply Lemma~\ref{Lem:ftc} with $L=v_n$ and rescaled by $n$, we have one of the following: If \eqref{ineq:Dint} occurs, then there are at least $\eps/2$ out of the $n/K$ strips of width $K$ where the slope of $p$ is negative. Hence, the path $p$ contains moves downwards along at least one vertical edge in the strip and the total amount of such edges for $p$ is at least $\eps n/(2K)$. On the other hand, if \eqref{ineq:slope} occurs, then at least $\eps/(2(v_n +\eps))$ of the $n/K$ strips have slope at most $K(v_n+\eps)$. 
\end{proof}

\subsection{Staying above the asymptote}\label{subsec:above_asymp}

We will now prove Proposition~\ref{Prop:above_min}. Construct a set of Bernoulli distributions $Ber(G,v)$ as follows:

For $v\ge 0$, choose $t(v)>t_0$ such that for $p(v):=G(t(v))$ we have $v\le (1-p(v))/p(v)$. Then let $Ber(G,v)$ be a probability distribution for the edges $e$ defined by:
\[
\Prob{ e=t_0}=p(v), \, \Prob{e=t(v)}=1-p(v).
\]
A $Ber(G,v)$-distributed weight $\tau'$ can be coupled with a $G$-distributed weights by setting
\begin{equation}\label{eq:coupling}
\tau'=t_0 \;\text{ if } \tau \in [t_0,t(v)], \quad \tau'=t(v) \; \text{ otherwise. }
\end{equation}
Let $\Lambda_{S(v,G)}$ be the time constant function in the directed model with distribution $Ber(G,v)$. We note the following from Theorem~\ref{Thm:exact}:

\begin{proposition}
\label{Prop:dir_asymp}
Let $G$ have no atom at $t_0$. Then $\Lambda_{S(v,G)}(v)>v+t_0$ for all $v\ge 0$. 
\end{proposition}

So set $\eps:=\Lambda_{S}(v)-(t_0+v)>0$ where we abbreviate $\Lambda_S=\Lambda_{S(G,v)}$.  We claim that for all $0\le w < v$ we have 
\begin{equation}
\label{ineq:epsa}
\Lambda_{S}(w)-(t_0+w)\ge \eps.
\end{equation}
Suppose not, then the distance of $\Lambda_{S}$ from the straight line is smaller at $w$ than at $v$, so the slope between $v$ and $w$ is greater than 1. But by convexity of $\Lambda_{S}$ this contradicts the fact that $\Lambda_S$ eventually coincides with a linear function of slope 1. 

This motivates the definition of the following subset of $[0,v]$:
\[
\J:=\{ w \in [0,v]: \, \Lambda_{S}(w)-\Lambda(w) \ge \eps/2 \}.
\]
For $w \in \J^c$, inequality \eqref{ineq:epsa} implies $\Lambda(w)\ge t_0+w+\eps/2  $. In particular, if $v\in J^c$, then we are done, so let us assume $v \in J$.

\begin{proposition}
\label{Prop:directed}
There exist constants $A_1,b_1>0$ such that for all $w \in J$ and $n\ge 1$
\[
\Prob{ \gamma^n(w) \text{ is directed } }\le A_1 \exp(-b_1 n). 
\]
\end{proposition}
\begin{proof}
Let $\mathcal M_{n,w}$ be the event $T((0,0),(n,\wn))\le (\Lambda(w)+\eps/4)n$, then by Theorem~\ref{Thm:upper_deviation}, there exists $A_2,B_2>0$ uniform for $w\in [0,v]$ such that $\Prob{\mathcal{M}_{n,w}^c}\le A_2 \exp(-B_2 n)$. So under the event $\mathcal M_{n,w}$, assume $\gamma^n(w)$ is directed. Then we have $T((0,0),(n,\wn))=T_{dir}((0,0),(n,\wn))\ge T_S((0,0),(n,\wn))$, where we use the coupling \eqref{eq:coupling}. Since $w \in J$, this means the event $\{T_S((0,0),(n,\wn))\le (\Lambda_S(w)-\eps/4)n\}$ occurs, which is also exponentially small by Theorem~\ref{Thm:lower_deviation_dir}. 
\end{proof}

Since $\Lambda$ is Lipschitz and $\Lambda_S$ has an explicit formula, we can compute $\delta>0$ such that $\Lambda_S(v')-\Lambda(v')\ge \eps/4$ for all $v'\in [v,v+\delta]$. Clearly, Proposition~\ref{Prop:directed} holds for such $v'$ as well with modified constants. 

We now construct a map from our edge-model with distribution $G$ to a macroscopic, standard FPP Bernoulli-distributed site model. 

Given $K$ to be determined later, let $B_K \subset \Z^2$ be the collection of vertices and edges contained in the half-open rectangle $[0,K)\times [0, 2\lceil K(v+\delta)\rceil)$, where we naturally embed the lattice in $\R^2$. By applying translations of the rectangle of the form $\theta_{ij}=(iK, 2j \lceil K(v+\delta)\rceil)$ with $i,j\in \Z$, we form a set of rectangles $\mathcal A_K$ partitioning the lattice such that each edge and vertex lies in a unique rectangle. Translating each rectangle by $\theta'=(0, \lceil K(v+\delta)\rceil)$ yields another partition $\mathcal A_K'$ with the same property. We say a semi-directed path $p$ \emph{trapezoidally crosses} $B_K$ if it contains a vertex in $\{0\}\times [0, \lceil K(v+\delta)\rceil)$ and $\{K\}\times [0, 2\lceil K(v+\delta)\rceil)$ and the subpath of $p$ between these vertices is completely contained in $B_K$. We adopt the terminology for translations of $B_K$. 

Let $T(B_K;h_1,h_2)$ be the passage time from $(0,h_1)$ to $(K,h_2)$ over all paths completely contained in $B_K$. Define the event $\mathcal Q_K$ that the trapezoidal crossing time is close to its infimum:
\[
\mathcal Q_K:=\left\{ \begin{array}{l}  \exists h_1,h_2: \,  \, 0\le h_1 < \lceil K(v+\delta)\rceil, \, 0\le h_2 \le h_1+\lceil K(v+\delta)\rceil:\\ T(B_K;h_1,h_2)< h_2-h_1+Kt_0+\min(2,K\eps/4) \end{array} \right\}.
\]
We generalize this to arbitrary $\mathbf x \in \mathcal A_K \cup \mathcal A_K'$.  Choose $i,j \in \Z$, $b\in \{0,1\}$ such that $\mathbf x=B_K+\theta_{i,j}+b \theta'$. And let $T(\mathbf x;h_1,h_2)$ be the passage time from $(iK, h_1)$ to $((i+1)K, h_2)$ over all paths completely contained in $\mathbf x$. We define the translated event
\[
\mathcal Q_k^{\mathbf x}:=\left\{ \begin{array}{l}  \exists h_1\in [(2j+b)\lceil K(v+\delta)\rceil, (2j+b+1)\lceil K(v+\delta)\rceil), \\ h_2 \in  [(2j+b)\lceil K(v+\delta)\rceil, h_1+\lceil K(v+\delta)\rceil] :\\ T(\mathbf x;h_1,h_2)< h_2-h_1+Kt_0+\min(2,K\eps/4) \end{array} \right\}.
\]
By disjointness of the rectangles, each collection of Bernoulli random variables $\{\1(\mathcal Q_k^{\mathbf x}), \, x \in \mathcal A_k \}$ and $\{\1(\mathcal Q_k^{\mathbf x}), \, x \in \mathcal A_k' \}$ is i.i.d.

\begin{proposition} \label{Prop:Gk} $\lim_{K\to\infty} \Prob{\mathcal Q_K} = 0.$\end{proposition}
\begin{proof}
For each value of $\Delta h:=|h_2-h_1|$, choose $w(\Delta h)\in [0,v+\delta]$ such that $\Delta h=\lceil Kw \rceil$. If $w\in J$, then the geodesic $\gamma$ from $(0,h_1)$ to $(K,h_2)$ is directed with probability at most $A_1 \exp(-b_1K)$ by Proposition~\ref{Prop:directed}.  And if $\gamma$ is not directed, it contains at least $\Delta h+2$ vertical edges. Hence with exponentially large proability and fixed $h_1,h_2$
\[
T(B_K;h_1,h_2)\ge T(\gamma)\ge h_2-h_1+Kt_0+2. 
\]
If $w\in J^c$, then $\Lambda(w)\ge t_0 + w +\eps/2$. Choose $K>8/\eps$, then the event $T(B_K;h_1,h_2)< h_2-h_1+Kt_0+K\eps/4$ implies $T(\gamma)\le K(t_0+w+3\eps/8)$, which is small by Theorem~\ref{Thm:lower_deviation_dir}. After taking a union bound over the order $K^2$ values for $h_1$ and $h_2$, the probability of the event $\mathcal Q_K$ is still decreasing in $K$. 
\end{proof}

\begin{figure}[h]
\begin{center}
\begin{minipage}[c]{0.8\textwidth}
\includegraphics[width=0.9\textwidth]{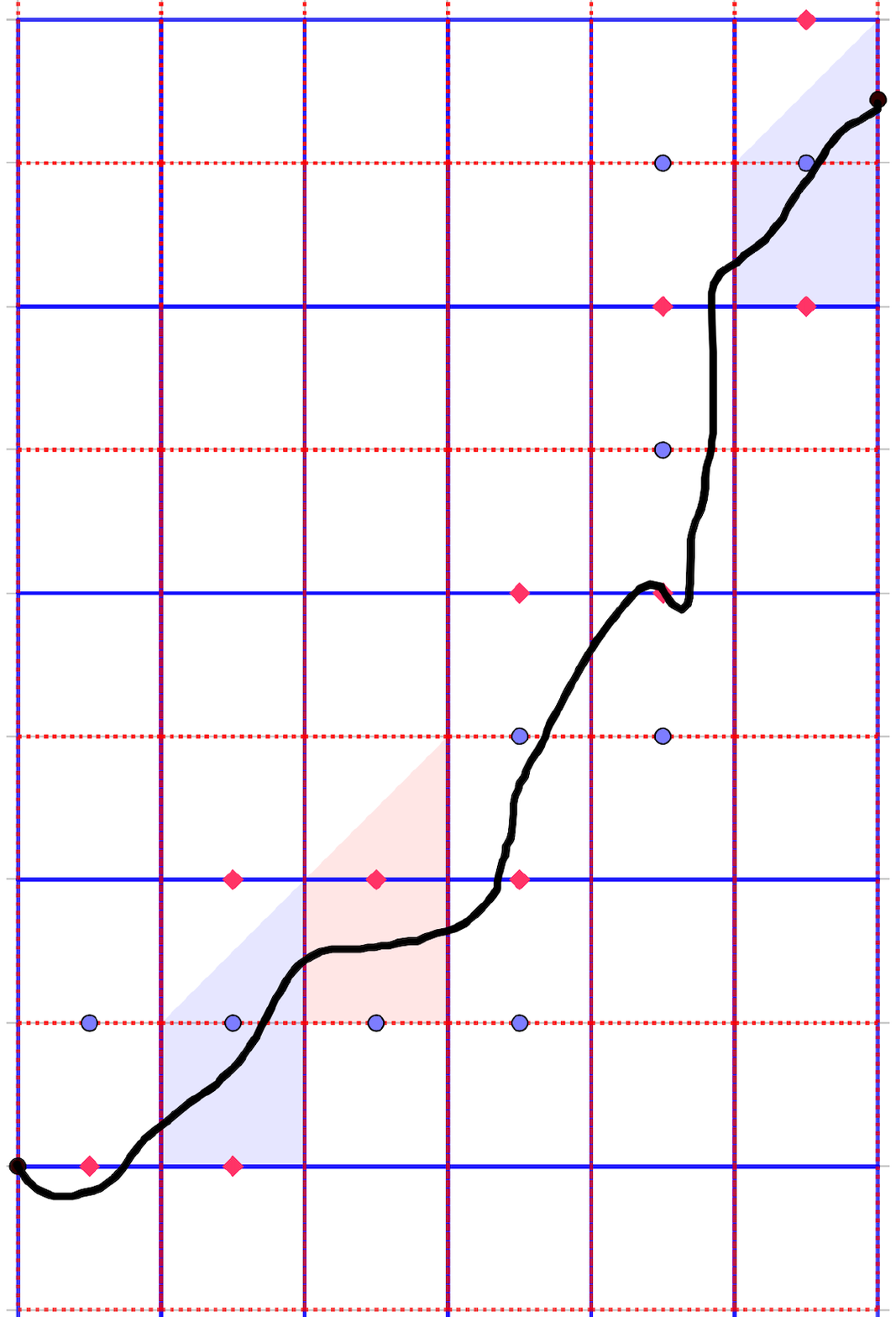}
\end{minipage}
\end{center}
\caption{A semi-directed path starting from the origin together with the associated macroscopic vertices from $\mathcal A_K$ (blue) and $\mathcal A_K'$ (red). Shaded regions indicate a trapezoidal crossing.}
\label{Fig:macro} 
\end{figure}

Consider a semi-directed path $\gamma$ from $(0,0)$ to $(n,\vn)$, which traverses the adjacent vertices $x_0=(0,0),x_1,x_2,...,x_N=(n,\vn)$ in this order. Let $\mathbf x_i$ and $\mathbf x_i'$ be the elements of $\mathcal A_K$ and $\mathcal A_K'$, respectively, such that $x_i \in \mathbf x_i \cap \mathbf x_i'$.  After erasing loops (i.e.\ removing all $\mathbf x_k$ with $i<k\le j$ and $\mathbf x_i=\mathbf x_j$), we then obtain the semi-directed macroscopic image paths $\boldsymbol\phi(\gamma)$ from $\mathbf 0:=B_K$ to $\mathbf x_N$ as well as the path $\boldsymbol{\phi}'(\gamma)\subset \mathcal A_K'$ from $\mathbf 0':=B_K-\theta$ to $\mathbf x_N'$. 

Let $p=\Prob{\mathcal Q_K}$ and we give each macroscopic vertex weight $0$ with probability $p$ and weight $1$ with probability $1-p$. We thus couple the microscopic model with a site percolation model on the macroscopic vertices $\mathcal A_K$ (see Figure~\ref{Fig:macro}).  

We will now make a case distinction on $\gamma$ when $\gamma$ is the geodesic and show that in each case $$T(\gamma)=T((0,0),(n,\vn))\ge (t_0+v+\eps_0)n$$ for some $\eps_0>0$ surely or with exponentially high probability in $n$. By Theorem~\ref{Thm:upper_deviation}, this yields Proposition~\ref{Prop:above_min}. 

By Corollary~\ref{Cor:ftc} with $\eps=\delta/2$, one of two things hold. Either 
\begin{equation}\tag{S1a}\label{S1a}
\mem{The path }\gamma \mem{ moves downward along at least } \delta n/(4K) \mem{vertical edges.}
\end{equation}
Then $T(\gamma)\ge (t_0+v+\delta/(2K))n$ as desired. Or
\begin{equation}\tag{S1b}\label{S1b}
\begin{split}
&\mem{There exists at least } \delta n/(4K(v+\delta/2)) \mem{ horizontal strips of width } K \\
&\mem{in which the slope is at most } v+\delta/2.
\end{split}
\end{equation}

We can split the statement \eqref{S1b} further into the two cases
\begin{align}
&\mem{The subpath of }\gamma \mem{ in half of the } \delta n/(4K(v+\delta/2)) \mem{ horizontal strips} \tag{S2a}\label{S2a} \\
&\mem{ is not directed.}\nonumber \\
&\mem{The subpath of }\gamma \mem{ in at least } \delta n/(8K(v+\delta/2)) \mem{ of the horizontal strips} \tag{S2b}\label{S2b}\\
&\mem{ is directed and of bounded slope.}\nonumber
\end{align}

If \eqref{S2a} holds, we obtain the bound $T(\gamma)\ge (t_0+v+1/(4K(v+\delta/2)))n$, because each subpath must have at least two additional vertical edges. Thus we are left with \eqref{S2b}, which we further divide into two cases according to the heights at which the path $\gamma$ enters a strip more frequently

\begin{equation}\tag{S3a}\label{S3a}
\begin{split}
&\mem{There exist at least } \delta n/(16K(v+\delta/2)) \mem{ elements of } \mathcal A_K' \\
&\mem{ that are trapezoidally crossed by } \gamma.
\end{split}
\end{equation}
\begin{equation}\tag{S3b}\label{S3b}
\begin{split}
&\mem{There exist at least } \delta n/(16K(v+\delta/2)) \mem{ elements of } \mathcal A_K \\
&\mem{ that are trapezoidally crossed by } \gamma.
\end{split}
\end{equation}

We continue with \eqref{S3b}, the argument for \eqref{S3a} will be analogous. Let $\mathcal T_K\subset \mathcal A_k$ denote the subset of rectangles trapezoidally crossed by $\gamma$ and we distinguish:

\begin{align}\tag{S4a}\label{S4a}
&\mem{For half of the elements }\mathbf x \in \mathcal T_K \mem{ the event } \mathcal Q_K^{\mathbf x} \emph{ does not occur.}\\
\tag{S4b}\label{S4b}
&\mem{For half of the elements }\mathbf x \in \mathcal T_K \mem{ the event } \mathcal Q_K^{\mathbf x} \emph{ occurs.}
\end{align}
By definition of $\mathcal Q_K$ , the event \eqref{S4a} yields $T(\gamma)\ge (t_0+v+min(2,K\eps/4)/(32K(v+\delta/2))n$. To understand \eqref{S4b}, we look at the macroscopic model and distinguish

\begin{equation}\tag{S5a}\label{S5a}
\begin{split}&\mem{The path } \boldsymbol\phi(\gamma) \mem{ moves downwards along at least } \delta n/(64K(v+\delta/2))\\& \mem{ macroscopic edges.}\end{split}
\end{equation}
\begin{equation}\tag{S5b}\label{S5b}
\begin{split}&\mem{The path } \boldsymbol\phi(\gamma) \mem{ moves downwards along at most } \delta n/(64K(v+\delta/2)) \\&\mem{ macroscopic edges.}\end{split}
\end{equation}

For each downward move of $\boldsymbol\phi(\gamma)$ , the microscopic path $p$ must move downwards along one edge. So if \eqref{S5a} holds, then $T(\gamma)\ge (t_0+v+\delta /(64K(v+\delta/2))n$. We conclude by showing that \eqref{S5b} is unlikely. Recall $x_n=(n,\vn)$, so the macroscopic coordinates of $\mathbf x_N$ are approximately $(n/K, vn/(2K(v+\delta))$ with a negligible rounding error in each coordinate. For the macroscopic passage time $\mathbf T(\boldsymbol\phi(\gamma))$  we then get the upper bound by counting the number of sites of $\boldsymbol\phi(\gamma)$ and subtracting all sites with weight $0$ where $\mathcal Q_K^{\mathbf x}$ occurs:
\begin{equation}
\label{ineq:low_macro}
\hspace*{-2mm}
\mathbf T(\boldsymbol\phi(\gamma))\le \left[v/(2(v+\delta))+\delta/(64(v+\delta/2))+(1-\delta/(32(v+\delta/2))) \right](n/K)
\end{equation}

Finally, we apply Lemma~\ref{Lem:spm} with $\eps=\delta/(2^7(v+\delta/2))$. For large enough $K$, the condition $p<p_\eps$ is ensured by Proposition~\ref{Prop:Gk}. We conclude that such a path $\boldsymbol\phi(\gamma)$ with very low passage time \eqref{ineq:low_macro} exists only with exponentially small probability. This finishes the proof of Proposition~\ref{Prop:above_min}.

\section{Proofs of auxiliary results}
\label{Sec:appendix}
Most of the results covered in this section are known for the standard FPP model. The proofs thus follow existing literature, though some arguments are made easier by the simplification of our model. 
\subsection{The time constant, proof of Theorem~\ref{Thm:time_const}}\label{ss:time_const}

The heart of the proof is the subadditive ergodic theorem, originally developed by Kingman. The version below is due to Liggett.

\begin{theorem}[The subadditive ergodic theorem (as stated in \cite{50years})] Let $(X_{m,n})_{0\le m < n}$ be a family of random variables that satisfies:

\begin{enumerate}[label=(\alph*)]
\item $X_{0,n}\le X_{0,m} + X_{m,n}$, for all $0<m<n$.
\item The distribution of the sequences $(X_{m,m+k})_{k\ge 1}$ and $(X_{m+1,m+k+1})_{k\ge 1}$ is the same for all $m\ge 0$.
\item For each $k\ge 1$, the sequence $(X_{nk,(n+1)k})_{n\ge 0}$ is stationary. 
\item $\E{}{X_{0,1}}<\infty$ and $\E{}{X_{0,n}}>-cn$ for some finite constant $c$. 
\end{enumerate}
Then 
\begin{equation}
\lim_{n\to \infty} \frac{X_{0,n}}{n}  \text{ exists a.s. and in }  L^1.
\label{eq:mvt}
\end{equation}
Furthermore, if the stationary sequence in (c) is also ergodic, then the limit in \eqref{eq:mvt} is constant almost surely and equal to 
\begin{equation}
\label{eq:mvt2}
\lim_{n\to \infty} \frac{\E{}{X_{0,n}}}{n} = \inf_n \frac 1 n \E{}{X_{0,n}}.
\end{equation}

\end{theorem}

\begin{lemma}[Exponential decay of passage times]
\label{Lem:expo_decay}
For any distribution $G$ and non-negative integers $u,v$ there exists a constant $C>0$ such that for all $k$
\[
\Prob{ T((0,0),(u,v))> k }\le C \,2^{-k}.
\]
\end{lemma}
Consequently, passage times have finite expectation.
\begin{proof}
Let $Y_1,...,Y_q$ be i.i.d.\ random variables with distribution $G$ and set $Y=Y_1+...+Y_u$. Then there exists $t\ge 0$ such that 
\[
\Prob{Y > t} \le \frac 1 2.
\]
For $h\in \Z$, let $\Sigma_h$ denote the sum of the edge weights of the $u$ edges along the horizontal path from $(0,h)$ to $(u,h)$. Then almost surely there exists a stopping time $T \ge 0$ such that $\min(\Sigma_T, \Sigma_{-T}) \le t$ and $T$ is minimal with that property. Note that
\[
 T((0,0),(u,v)) \le t + 2 T + v. 
\]
Since edge weights are independent, we obtain:
\begin{align*}
&\Prob{T((0,0),(u,v))> k} \le \Prob{ T > \frac{k-t-v}{2}}\\
&=\Prob{X_i > t, \, \forall \, |i|\le  \frac{k-t-v}{2}} \le \frac{1}{2^{k-(t+v+1)}}.
\end{align*}
\end{proof}

We first show the existence of a time constant for rational $v$. 

\begin{lemma}
\label{Lem:limit_multiple}
Consider rational $v=p/q$ with $p \in \Z{\ge 0}, q \in \Z^+$ and for $0\le m < n$ we define:
\[
X_{m,n}:=T((qm, pm), (qn, pn))=T((qm, \lceil vqm\rceil), (qn,  \lceil vqn\rceil)).
\]
Then \eqref{eq:mvt} and \eqref{eq:mvt2} hold. 
\end{lemma}
\begin{proof}
We apply the subadditive ergodic theorem. We established that the passage times are achieved by geodesics and the composed path of the geodesic from $(0,0)$ to $(qm, pm)$ and of the geodesic from $(qm, pm)$ to $(qn, pn)$ cannot have a passage time less than the passage time of the geodesic from the origin to $(qn, pn)$. This shows (a). Since the horizontal passage times are i.i.d. the sequence $(X_{nk,(n+1)k})_{n\ge 0}$ is stationary ergodic and (b) holds as well. We clearly have $X_{0,n}\ge 0$ and $\E{}{X_{0,1}}<\infty$ follows from Lemma~\ref{Lem:expo_decay}.\end{proof}

Next, we show that the limit 
\begin{equation}
\label{eq:aslimit}
\lim_{n\to \infty} \frac 1 n T((0,0),(n,\vn))
\end{equation}
exists almost surely. For each $n$, define $u(n)$ as the smallest integer larger or equal to $n$ which is a multiple of $q$. By subadditivity:
\[
| T((0,0),(n,\vn)) - T((0,0), (u(n), v u(n)))| \le T((n,\vn),  (u(n), v u(n)))=: Z_n. 
\]
Note that $Z_n$ equals in distribution the passage time from the origin to a point inside the rectangle $[0,...,q-1]\times [0,...,vq]$. Hence, by applying Lemma~\ref{Lem:expo_decay}, we find a constant $C>0$ uniform in $n$ such that for all $k$
\[
\Prob{ Z_n > k} \le C 2^{-k}.
\]
By Borel--Cantelli, we thus have that almost surely
\[
\lim_{n \to \infty} \frac{Z_n}{n} = 0. 
\]
Together with Lemma~\ref{Lem:limit_multiple} this shows \eqref{eq:aslimit}. The estimate also implies convergence in expectation. 

Moreover, since
\[
|T((0,0),(n,\vn))-T((0,0), (n,\wn))|\le |\vn-\wn|
\]
we have that the time constant is uniformly Lipschitz and thus continuously extends to the reals
\[|\Lambda(v)-\Lambda(w)|\le |v-w|.
\]
Hence, Theorem~\ref{Thm:time_const} holds for all $v \in [0,\infty)$. 

\subsection{Proof of Lemma~\ref{Lem:positive_time}} 

\begin{lemma}
\label{Lem:pathcounting}
Given integers  $M \ge k \ge 1$, let $S(M,k)$  be the number of distinct $k$-tuples $(a_1,...,a_k)$ of integers whose absolute values sum up to $M$. We then have that
\[
S(M,k)= \sum_{\rho=0}^{M-1} \binom{k}{\rho}\binom{M-1}{\rho}2^{\rho}+\binom{k}{M}2^{M}
\]
\end{lemma}
\begin{proof}
Let $\rho \in \{0,...,M\} $ be the number of $a_i$ that are non-zero. The problem reduces to count the number of distinct tuples $(b_1,...,b_\rho)$ of positive integers that sum up to at most $M$. We multiply this number with the number of choices to select $\rho$ indices from $\{1,...,k\}$ and the signs of the $a_i$. 

We use a stars and bars argument to show that the number of values for the $b_i$ is indeed $\binom{M-1}{\rho}$ for $\rho<M$. So given $M$ stars, select $\rho$ bars out of the $M-1$ possible positions between the stars. The number of stars between two consecutive bars then corresponds to $b_1,...,b_\rho,M-\sum_i b_i$, respectively.
\end{proof}
An application of this lemma is a bound on the number of semi-directed paths with plausible passage time. 
\begin{corollary}
\label{Cor:countpaths}
Let $G$ be arbitrary and $C\ge 1$. There are at most $(2Ce)^{n+1}$ semi-directed paths from the origin to a point $(n,m)$ with passage time at most $Cn$. 
\end{corollary}
\begin{proof}
A semi-directed path is defined by its vertical jumps and the number of vertical edges is a lower bound for the passage time. We then apply the lemma and a standard inequality for the binomial coefficient. 
\end{proof}

Since $\Lambda(v)\ge v+t_0(G)$, we are left with the case $v=0$ and $t_0=0$. Clearly, if $G=\delta_x$, then $\Lambda(0)=x$. So assume by contradiction that $G$ is not deterministic and $\Lambda(0)=0$. Then there exists $\tau_1>\tau_0(G)$ such that $p:=G( [\tau_0,\tau_1))\in (0,1)$. Let us now look at all possible geodesics $\gamma(0,(n,0))$ and show that it is very unlikely for them to have a low passage time. 

Let $L_\eps$ be the event that the geodesic from $(0,0)$ to $(n,0)$ contains at most $2\eps n$ vertical edges and at most $\eps n$ horizontal edges with passage time $\tau_1$ or higher.  By Theorem~\ref{Thm:time_const}, it suffices to take a union bound over all semi-directed paths $p$ that satisfy these properties. 

Since $p$ is defined by its vertical jumps, we apply Lemma~\ref{Lem:pathcounting} to count all semi-directed paths with at most $2\eps n$ vertical edges. We then multiply this with the number of ways to choose $(1-\eps)n$ edges with weights less than $\tau_1$. Each instance of a path and edge weight configuration then has probability at most $p^{(1-\eps)n}$. Thus, we get (dropping rounding to integers for the sake of readability)
\begin{align*}
\Prob{E_\eps} &\le S(n, 2\eps n) \binom{n}{\eps n} p^{(1-\eps)n}
\le \sum_{\rho=0}^{2\eps n} \binom{n}{\rho}\binom{2\eps n}{\rho}2^\rho \cdot \binom{n}{\eps n}p^{n/2} \\
&\le (2 \eps n+1) \binom{n}{2 \eps n}\binom{2\eps n}{\eps n}2^{\eps n} \binom{n}{\eps n}p^{n/2} 
\le n (2e^3 \eps^{-3})^{\eps n} p^{n/2}.
\end{align*}
As $\eps\to 0$, we have $(2e^3 \eps^{-3})^{\eps}\to 1 < p^{-1/2}$. So for small enough $\eps$, the expression above goes to zero as $n\to \infty$ and hence we must have had $\Lambda(0)>0$. 

\subsection{Continuity, proof of Lemma~\ref{Lem:continuity}}\label{ss:cont}
\begin{lemma}
\label{Lem:beta}
\looseness -1
Let $G$ have unbounded support and let $\rho(\gamma^n, B)$ be the number of horizontal edges with weights greater than $B$ in the geodesic $\gamma^n:=\gamma((0,0),(n,\vn))$. For all $\beta>0$ there exists  $B>0$ such that $\lim_{n \to \infty} \Prob{\rho(\gamma^n,B)>\beta n}=0$. 
\end{lemma}
\begin{proof}
Define new edge weights $e'_i$ for the lattice where $e'_i=1$ if the original edge weight $e_i$ is at most $B$ and set $e'_i=0$ otherwise. So if the event $\rho(\gamma^n,B)>\beta n$ occurs, then the passage time in the modified model is at most $(v+1-\beta)n$. However by Lemma~\ref{Lem:cnt_ber}, there exists $B$ such that the time constant in this Bernoulli model is at least $v+1-\beta/2$. The lemma then follows from Theorem~\ref{Thm:time_const}.
\end{proof}

Now let us turn to the proof of Lemma~\ref{Lem:continuity}. In FPP the proof of continuity goes back to Cox and Kesten \cite[Lemma 2]{coxkesten}. Let $\eps>0$. Our goal is to show that for some $B>0$ the passage times of $G^B$ and $G$ to $(n,\vn)$ differ by at most $\eps n$ for large $n$ and with a probability uniformly bounded away from zero. Then convergence of the time constant follows from Theorem~\ref{Thm:time_const}. 
We do a coupling of an edge weight environment $\{\tau_i\}$ with distribution $G$ with one with distribution $G^B$ by defining $\tau_i':=\min(\tau_i, B)$. Let $T(p)$ denote the passage time in the original environment and $T^B(p)$ the passage time with the $\tau_i'$ weights instead. Let $\gamma$ and $\gamma^B$ be the geodesic from the origin to $(n,\vn)$, respectively in the original and the truncated environment. 

For any semi-directed path $p$ with horizontal edges $e_1,..,e_n$ we construct a modified path $d(p)$ as follows: For any edge $e_i$ with $\tau(e_i)\ge B$, replace this edge by the detour of $e_i$ until we find an edge with weight strictly less than $B$. The passage time then increases by at most $2k_{B}(e_i)$. 

Note that
\[
T^B(\gamma^B)\le T^B(d(\gamma^B))=T(d(\gamma^B)), \text{and } T^B(\gamma^B)\le T(\gamma)\le T(d(\gamma^B)).
\]
So it suffices to show that the probability of the event $T^B(d(\gamma^B))-T^B(\gamma^B)>\eps n$ goes to zero as $n\to\infty$. Let $\mathcal C$ be the set of all semi-directed paths to $(n,\vn)$ that have at most $2\Lambda(v)n$ vertical edges. By Theorem~\ref{Thm:time_const}, the event $\gamma^B \not \in \mathcal C$ goes to zero for large $n$. 
So it suffices to take a union bound over all paths $p\in \mathcal C$. For any such path $p$ and $i \in \{1,...,n\}$, set $a_i=0$ if $\tau^B(e_i)<B$ and $a_i=2|k_B(e_i)|$ otherwise. 
Define $\mathcal A$ as the set of tuples $(a_1,...,a_n)$ of non-negative numbers satisfying the following conditions:
\begin{itemize}
\item $S:=\sum a_i \ge \eps n$
\item $ \rho:=| i : \,  a_i \ge 1 | \le \beta n := \eps n/2$.
\end{itemize}
We may assume the second condition because of Lemma~\ref{Lem:beta}. Define $q=G([B,\infty))$. To count the set, we use the same combinatorial argument as in Lemma~\ref{Lem:pathcounting} to count all ways to sum $\rho$ positive integers to (at most) $S$ and multiply them by the number of ways to choose $\rho$ of $n$ indices with $a_i\ge 1$. A union bound then yields:
\begin{align*}
&\Prob{T^B(d(\gamma^B))-T^B(\gamma^B)>\eps n, \gamma^B \in C, \rho(\gamma^B,B)\le \beta n }
\le \sum_{p \in \mathcal C} \sum_{\mathcal A} \prod_{i=1}^n q^{a_i} \\
&\le \sum_{p \in \mathcal C} \sum_{S=\eps n}^\infty \sum_{\rho=1}^{\beta n} \binom{n}{\rho} \binom{S-1}{\rho} q^S 
\le \sum_{p \in \mathcal C} \sum_{S=\eps n}^\infty (\beta n) \binom{n}{\beta n}\binom{S}{\beta n} q^S \\
&\le \sum_{p \in \mathcal C} \beta n \left(\frac{e}{\beta n}\right)^{\beta n} \sum_{S=\eps n}^\infty S^{\beta n} q^S 
\le   \sum_{p \in \mathcal C} \beta n \left(\frac{e}{\beta n}\right)^{\beta n} (\beta n)! \sum_{S=\eps n}^\infty  (qe)^S \\
&\le   \sum_{p \in \mathcal C} 6 \beta n^2  \sum_{S=\eps n}^\infty  (qe)^S \le (4\Lambda(v)e)^{n+1}\frac{6 \beta n^2}{qe(1-qe)}  (qe)^{\eps n}
\end{align*}
We applied Corollary~\ref{Cor:countpaths} in the last step. We conclude that the probability goes to zero if $q<(4\Lambda(v)e^2)^{-1/\eps}$ which is guaranteed for $B$ large enough. 

\subsection{Right-tail large deviations, proof of Theorem~\ref{Thm:upper_deviation}}\label{ss:right_dev}

This theorem and its rather simple proof originate from Grimmett and Kesten \cite[Theorem~3.2]{grimmettkesten}. To simplify the notation, define $t_n(v)=T((0,0),(n,\vn))$. Given $\eps>0$ and $k\in \{0,...,\lceil 3v_0/\eps\rceil\}$, by Theorem~\ref{Thm:time_const} there exists $N(k)$ such that
\[
\frac 1 m \E{}{t_m(k\eps/2)}\le \Lambda(k\eps/3)+\eps/3 \quad \forall \, m\ge N(k).
\]
Set $N:=\max_k N(k)$. Since $|t_n(v)-t_n(w)|\le |v-w|n+1$, we can conclude 
\begin{equation}
\label{ineq:tnv}
\frac 1 N \E{}{t_N(v)}\le \Lambda(v)+\eps \quad \forall \, v\in [0,v_0].
\end{equation}
Let us assume $n=rN$ for some integer $r$. Define $Y_i=T((N(i-1), \lceil v N \rceil (i-1)), (Ni, \lceil v N \rceil i))$ for $i=1,...,r$. These variables are independent and distributed like $Y_1=t_N(v)$. Also let $Z_i=Y_i-\E{}{Y_i}$. Using subadditivity and \eqref{ineq:tnv} we conclude
\begin{align*}
\Prob{T((0,0),(n, \lceil v N \rceil r))> n (\Lambda(v)+2\eps)} &\le \Prob{Y_1+...Y_r > n(\Lambda(v)+2\eps)} \\
\le \Prob{Z_1+...+Z_r> rN\eps} &\le \left( \frac{\E{}{e^{\xi Z_1}}}{e^{N\eps\xi}}\right)^r
\end{align*}
where $\xi>0$. By Lemma~\ref{Lem:expomoment},  $\E{}{e^{\xi t_1(0)}}<\infty$ for sufficiently small $\xi$. Since $t_N(v)$ is stochastically dominated by $(v_0+t_1(0))N$, we have that $\E{}{e^{\xi Y_1}}$ is uniformly bounded in $v$. Since also $\E{}{Z_1}=0$ we have $\E{}{e^{\xi Z_1}}=1+o(\xi)$. Hence, choosing $\xi$ small enough yields $\E{}{e^{\xi Z_1}}<e^{N\eps\xi}$ and the probability above decreases exponentially in $n$. 

Finally, for general $n' \in \N$, there exists $n\le n'$ such that $n'-n<N$. By subadditivity,
\begin{align*}
&\Prob{t_n(v)>n(\Lambda(v)+3\eps)}\\ \le   \; & \Prob{T((0,0),(n, \lceil v N \rceil r))> n (\Lambda(v)+2\eps)}+\Prob{T((n, \lceil v N \rceil r), (n',\lceil v'n\rceil))>\eps n}
\end{align*}
The two points of the summand on the right lie in a box of constant size $N\times (N+2)v_0$. By Lemma~\ref{Lem:expomoment} and a Chernoff bound, the probability of their passage time being large also decreases exponentially in $n$. 

\subsection{Concentration inequalities}\label{ss:concentration}

While Theorem~\ref{Thm:lower_deviation_dir} can be shown with a standard but somewhat tedious block argument going back to Kesten, we use a shorter and more modern entropy argument to show a concentration bound first from which the theorem will easily follow.

\begin{theorem}[Concentration bound]
\label{Thm:concentration}
Let $G$ have finite variance. Then there exists a constant $C>0$ such that for all $x=(x_1,x_2) \in \Z^d$ and $t\ge 0$:
\[\Prob{|T(0,x)-\E{}{T(0,x)}|\ge t \sqrt{x_1}}\le e^{-Ct^2}.
\]
\end{theorem}

A central object is entropy. For a non-negative random variable $X$ we define

\[
\Ent\, X =\E{}{X \log X}- \E{}{X}\log \E{}{X}.
\]

Our goal is to show the following claim:

\textbf{Claim 2:} Write $T=T(0,x)$. There exists some constant $C>0$ such that for all $\lambda \in \R$
\begin{equation}
\label{ineq:cl}
\Ent\, e^{\lambda T}\le C \|x\|_1 \lambda^2 \E{}{e^{\lambda T}}. 
\end{equation}

The following arguments can be found in \cite{50years}. We will treat the $\lambda<0$ case, the argument for $\lambda>0$ is nearly identical and can be found  in \cite[Section~3.4.2]{50years}. First, we prove Theorem~\ref{Thm:concentration} using this claim. We use the Herbst argument. Set $\psi(\lambda)=\log \E{}{e^{\lambda(T-\E{}{T}}}$. One can check that
\[
\frac{\Ent\, e^{\lambda T}}{\lambda^2 \E{}{e^{\lambda T}}}=\frac d {dy} \left(\frac{\psi(\lambda)}{\lambda}\right).
\]
Integrating \eqref{ineq:cl} then yields
\[
\psi(\lambda)\le C \|x\|_1 \lambda^2. 
\]
So for $\lambda <0$ one has using Markov's inequality
\begin{align*}
&\Prob{T-\E{}{T}\le -t\sqrt{\|x\|}}=\Prob{e^{\lambda(T-\E{}{T})}\ge e^{-t\lambda\sqrt{\|x\|}}}\\
&\le e^{\psi(\lambda)+t\lambda\sqrt{\|x\|}} \le e^{C\|x\|\lambda^2+t\lambda\sqrt{\|x\|}}. 
\end{align*}
Choosing $\lambda=-t/(2C \sqrt{\|x\|})$ then completes the proof. 
 
Now let us prove the claim. Enumerate the \emph{horizontal} edge weights of $\Z^2$ by $\tau_{e_1},\tau_{e_2},...$. Let $\mathbb{E}_{i}$ be the expectation operator conditioned on $\{\tau_{e_j}\}_{j\neq i}$. For any measurable, non-negative function $X$ of the edge weights we define the conditional entropy
\[
\Ent_i\, X =\E{i}{X \log X}- \E{i}{X}\log \E{i}{X}.
\]
By tensorization of entropy (see \cite[Theorem~4.22]{concentration}), we have
\begin{equation}
\label{eq:tensor}
\Ent\, e^{\lambda T}\le \sum_{i=1}^\infty \E{}{\Ent_{i}\, e^{\lambda T}}.
\end{equation}

\begin{lemma}[Symmetrized LSI, \cite{concentration}]
Let $q(x)=x(e^x-1)$. Let $Z$ and $Z'$ be two i.i.d.\ random variables. Then for all $\lambda \in \R$,
\[
\Ent\, e^{\lambda Z} \le \E{}{e^{\lambda Z}q(\lambda(Z'-Z)_+)}.
\]
\end{lemma}

Together with \eqref{eq:tensor} this lemma yields:
\begin{equation}
\label{eq:lsi}
\Ent\, e^{\lambda T}\le \sum_{i=1}^\infty \E{}{e^{\lambda T}q(\lambda(T_i'-T)_+)}
\end{equation}
where $T_i'$ is the passage time at $e_i$ is replaced with an independent copy $\tau_{e_i}'$. Note that if $\gamma$ is a semi-directed geodesic from $0$ to $x$ for the original weights, then $T_i'-T>0$ only if $e_i \in \gamma$. Since $q(x)$ is increasing in $|x|$, we may apply the bound $q(\lambda(T_i'-T)_+)\le q(\tau_{e_i}')\1(e_i \in \gamma)$. In any of the directed models, the geodesic contains exactly $x_1$ horizontal edges. And by independence,
\begin{equation}
\Ent\, e^{\lambda T}\le \E{}{q(\lambda \tau_{e_i})}\E{}{e^{\lambda T}}x_1 \le \E{}{\tau_{e}^2}\lambda^2 \E{}{\lambda T}.
\end{equation}

\subsection{Left-tail large deviations, proof of Theorem~\ref{Thm:lower_deviation_dir}}\label{ss:left_dev}

Since the expected passage time $E_n(v):=\E{}{T(0,(n,\vn))}$ is Lipschitz as a function of $v$, we have that $E_n(v)/n$ converges uniformly to $\Lambda(v)$ for $v\in [0,v_0]$. So given $\eps>0$ and sufficiently large $n$, we have
\[
|E_n(v)-\Lambda(v)n|\le \frac \eps 2 n \quad \forall v \in  [0,v_0].
\]
Setting $t=\eps \sqrt{n}/2$ in Theorem~\ref{Thm:concentration} then yields the estimate, given that $G$ has finite variance. For arbitrary distributions, let $T_B$ be the passage time where the $G$-distributed weights are truncated at $B$. Under the usual coupling, we have that
\[
T_B(0,x)\le T_B(\gamma(0,x))\le T(\gamma(0,x))=T(0,x).
\]
Given $\eps>0$, choose $B$ large enough such that $\Lambda(v)\le \Lambda_B(v)+\eps/2$ for all $v\in [0,v_0]$. Then the event $T(0,(n,\vn))\le (\Lambda(v)-\eps)n$ implies $T_B(0,(n,\vn))\le (\Lambda(v)-\eps/2)n$, which has small probability.

\subsection{The site percolation model, proof of Lemma~\ref{Lem:spm}}\label{ss:spm}

The idea is to couple the site percolation model with the original FPP model with Bernoulli-distributed edge weights. Let us consider a weight configuration of the site percolation model with parameter $p$. We define a new percolation model on the edges of $\Z^2$ where the weight of an edge is set as the minimum of the weights of the two adjacent sites. Denoting the passage time of this model by $T_*^p(\cdot,\cdot)$, we note that $T_*^p(0,x)\le T_{site}^p(0,x)$. The edges $e$ in the new model are Bernoulli distributed such that $\Prob{\tau(e)=1}=(1-p)^2$. Two edge weights are dependent if and only if they neighbor the same site. We now apply a result by Liggett et.\ al.\ \cite{liggett} to conclude that this dependent edge weight measure stochastically dominates a product measure. In particular, \cite{liggett} implies that 
\[
\Prob{ T^p_{site}(0,x)<N} \le \Prob{ T_{*}^{p}(0,x)<N} \le \Prob{ T_{FPP}^{p^*(p)}(0,x)<N}
\]
where $p^*(p)\to 0$ as $p \to 0$. The lemma now follows from the continuity of the time constant \cite{cox} and large derivation results \cite{kesten} for standard FPP.

\subsection{The limit shape, proof of Theorem~\ref{Thm:limit_shape}}\label{ss:limit_shape}

Usually we define the time constant as $\mu(x)=\lim_{n\to \infty} T(0,nx)/n$ for $x=(x_1,x_2) \in \N^2$. Thus we have the relation $\mu(x)=x_1 \Lambda(x_2/x_1)$. Let $\theta=\arg x$. If convergence of the sets $B_t$ holds, then the boundary of the limit shape is the set $\{x\in \R: \mu(x)=1\}$ which translates to $\cos(\theta)\Lambda(\tan(\theta))=1$. 

Suppose the Shape Theorem does not hold, then there exists $\delta>0$ such that we can find infinity many $x=(x_1,x_2) \in \Z_{>0} \times \Z_{\ge 0}$ such that with positive probability
\[|T(0,x)-\Lambda(x_2/x_1)x_1|>\delta \|x\| .\]

Choose $v_0$ such that $\Lambda(v_0)\le v_0+t_0+\delta/5$ and $v_0>2((\Lambda(0)-t_0)\delta)^{-1}$. Let $v$ be a limit point of $x_2/x_1$. If $v\le v_0$, we get a contradiction by our large deviation results. If $v>v_0$, we must have $\Lambda(x_2/x_1)\le x_2/x_1+t_0+\delta/4$ for large $x$. Using subadditivity and a large deviation estimate, we have almost surely $x_2+t_0x_1 \le T(0,x) \le T(0,(x_1,0)+x_2 \le (\Lambda(0)+\delta/4)x_1+x_2$. Bringing the two bounds together yields a contradiction:
\begin{align*}
&|T(0,x)-\Lambda(x_2/x_1)x_1|\le |T(0,x)-x_2-t_0x_1|+\delta/4\, x_1 \\
&\le  (\Lambda(0)-t_0+\delta/4)x_1+\delta/4\, x_1\le \delta/2 (x_2+x_1) <\delta \|x\|.
\end{align*}

\begin{center}
 \section*{{\normalsize{Acknowledgements}}}
\end{center}
Special thanks go to my advisor Philippe Sosoe for giving helpful comments about my work and always listening to my ideas. I also want to thank Douglas Dow, whose work with Yuri Bakhtin inspired me to pursue this project, for providing feedback and for reaching out to me during the $51^{st}$ Saint-Flour Probability Summer School. So I should also show my appreciation for the organizers of said event, which has a long history \cite{kesten} of supporting the field of FPP. 

\end{document}